\documentclass[fleqn]{zzz}
\usepackage{mathrsfs}
\usepackage{listings}
\usepackage{xcolor}
\usepackage{tikz}
\usepackage{amsmath,amsthm,amsbsy,amssymb}
\usepackage{graphicx}
\usepackage{rotating}
\usepackage{fixmath}
\usepackage{hyperref}
\usepackage{stmaryrd}
\usepackage{relsize}
\usepackage{amssymb}
\usetikzlibrary{decorations}

\usepackage{refcheck}
\usepackage{enumerate}
\usepackage{amsmath}
\usepackage{stmaryrd}
\usepackage{booktabs}
\usepackage{textcomp}

\lstdefinestyle{mystyle}{
    backgroundcolor=\color{white},   
    commentstyle=\color{green},
    keywordstyle=\color{red},
    numberstyle=\tiny\color{orange},
    stringstyle=\color{red},
    basicstyle=\ttfamily\footnotesize,
    breakatwhitespace=false,         
    breaklines=true,                 
    captionpos=b,                    
    keepspaces=true,                 
    numbers=left,                    
    numbersep=5pt,                  
    showspaces=false,                
    showstringspaces=false,
    showtabs=false,                  
    tabsize=2
}
\pgfdeclaredecoration{dashsoliddouble}{initial}{
 \state{initial}[width=\pgfdecoratedinputsegmentlength]{
 \pgfmathsetlengthmacro\lw{.3pt+.5\pgflinewidth}
 \begin{pgfscope}
 \pgfpathmoveto{\pgfpoint{0pt}{\lw}}%
 \pgfpathlineto{\pgfpoint{\pgfdecoratedinputsegmentlength}{\lw}}%
 \pgfmathtruncatemacro\dashnum{%
 round((\pgfdecoratedinputsegmentlength-3pt)/6pt)
 }
 \pgfmathsetmacro\dashscale{%
 \pgfdecoratedinputsegmentlength/(\dashnum*6pt + 3pt)
 }
 \pgfmathsetlengthmacro\dashunit{3pt*\dashscale}
 \pgfsetdash{{\dashunit}{\dashunit}}{0pt}
 \pgfusepath{stroke}
 \pgfsetdash{}{0pt}
 \pgfpathmoveto{\pgfpoint{0pt}{-\lw}}%
 \pgfpathlineto{\pgfpoint{\pgfdecoratedinputsegmentlength}{-\lw}}% 
 \pgfusepath{stroke}
 \end{pgfscope}
 }
}

\hypersetup{
 colorlinks = true,
 urlcolor = blue!60!black,
 linkcolor = red!60!black,
 citecolor = green!60!black
}

\newtheorem{thm}[lemma]{Theoram}

\newtheorem{rem}[lemma]{Remark}
\newtheorem{lem}[lemma]{Lemma}
\newtheorem{prop}[lemma]{Proposition}
\newtheorem{defn}[lemma]{Definition}

\renewcommand{\arraystretch}{1.2}
\numberwithin{equation}{section}
\numberwithin{corollary}{section}
\numberwithin{remark}{section}
\numberwithin{theorem}{section}
\numberwithin{lemma}{section}

\newcommand{\Fhyp}[5]{\,\mbox{}_{#1}F_{#2}\hspace{-0.2mm}
\left(\hspace{-2mm}\begin{array}{c} {#3}\\ {#4}\end{array}
\hspace{-1mm};#5\right)}
\newcommand{\qhyp}[5]{\,\mbox{}_{#1}\phi_{#2}\hspace{-1mm}
\left(\hspace{-2mm}\begin{array}{c} {#3}\\{#4}\end{array}
\hspace{-1mm};#5\right)}
\newcommand{\dln}{\nabla\!{}_n}
\newcommand{\dlpa}[1]{\nabla\!{}_{{#1}}}

\makeatletter
\def\eqnarray{\stepcounter{equation}\let\@currentlabel=\theequation
\global\@eqnswtrue
\tabskip\@centering\let\\=\@eqncr
$$\halign to\displaywidth\bgroup\hfil\global\@eqcnt\z@
$\displaystyle\tabskip\z@{##}$&\global\@eqcnt\@ne
\hfil$\displaystyle{{}##{}}$\hfil
&\global\@eqcnt\tw@ $\displaystyle{##}$\hfil
\tabskip\@centering&\llap{##}\tabskip\z@\cr}

\def\endeqnarray{\@@eqncr\egroup
\global\advance\c@equation\m@ne$$\global\@ignoretrue}

\def\@yeqncr{\@ifnextchar [{\@xeqncr}{\@xeqncr[5pt]}}
\makeatother

\usepackage{scalerel,url}
\let\svus_
\catcode`_=\active %
\def_{\ifmmode\expandafter\svus\expandafter\bgroup\expandafter\lowerit
\else\svus\fi}
\def\lowerit#1{\ThisStyle{\raisebox{-2\LMpt}{$\SavedStyle#1$}}\egroup}
\makeatletter
\AtBeginDocument{
\check@mathfonts
\fontdimen13\textfont2=3.5pt
\fontdimen14\textfont2=3.5pt
\fontdimen16\textfont2=2.5pt
\fontdimen17\textfont2=2.5pt
 }
\makeatother

\begin{document}

\renewcommand{\PaperNumber}{***}

\FirstPageHeading

\ArticleName{The Laguerre constellation of classical orthogonal  Polynomials}
\ShortArticleName{The Laguerre constellation of classical orthogonal  Polynomials}

% Authors, for the paper (add full first names)
\Author{R. S. Costas-Santos$^{1}\href{https://orcid.org/0000-0002-9545-7411}{\includegraphics[scale=0.1]{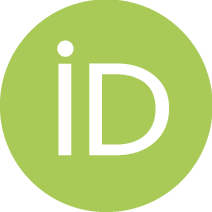}}{}$}

\AuthorNameForHeading{R. S. Costas-Santos}

\Address{$^1$ Department of Quantitative Methods, 
Universidad Loyola Andaluc\'ia, E-41704, Sevilla, Spain} 
% Address of First Author
\URLaddressD{
\href{http://www.rscosan.com}
{http://www.rscosan.com}
}
\EmailD{rscosa@gmail.com} % E-mail address of First Author

\ArticleDates{Received {\bf \today} in final form ????; 
Published online ????}
%%%%%%

% Abstract (Do not insert blank lines, i.e. \\) 
\abstract{A linear functional $\bf u$ is classical if there exist polynomials, 
$\phi$ and $\psi$, with $\deg \phi\le 2$, $\deg \psi=1$, such that 
${\mathscr D}\left(\phi(x) {\bf u}\right)=\psi(x){\bf u}$, 
where ${\mathscr D}$ is a certain differential, or difference, operator. 
The polynomials orthogonal with respect to the linear functional ${\bf u}$ are 
called {\sf classical orthogonal polynomials}.
In the theory of orthogonal polynomials, a correct characterization 
of the classical families is of great interest. In this work, on the one hand, 
we present the Laguerre constellation, which is formed by all the classical 
families for which $\deg \phi=1$, obtaining for 
all of them new algebraic  identities such as structure formulas, orthogonality properties 
as well as new Rodrigues formulas; on the other hand, we present a theorem that 
characterizes the classical families belonging to the Laguerre constellation.
}

% Keywords
\Keywords{Recurrence relation; 
Characterization Theorem;
Classical orthogonal po\-ly\-no\-mials;
Laguerre constellation;} 

% The fields PACS, MSC, and JEL may be left empty or commented out if not applicable
\Classification{42C05; 33C45; 33D45}
%%%%%%%%%%%%%%%%%%%%%%%%%%%%%%%%%%%%
\section{Introduction} \label{sec1}
%%%%%%%%%%%%%%%%%%%%%%%%%%%%%%%%%%%%
Orthogonal polynomials $(p_n(x))_n$ associated with a measure 
on the real line, i.e., 
\[
\int_{\mathbb R}
p_m(x)p_n(x) d\mu(x)=d^2_n \, \delta_{mn},
\] 
satisfy a three-term recurrence equation
\begin{equation} \label{ttrr}
xp_n(x) = \alpha_n p_{n+1}(x)+\beta_n p_n(x)+\gamma_n p_{n-1}(x),
\end{equation} 
where $p_{-1}(x) =0$, $p_0(x)=1$ and, according to the {\sf Favard 
theorem} (cf. \cite[p. 21]{mr0481884}), if $\gamma_n\ne 0$ for all 
$n\in \mathbb N$ this recurrence 
completely characterizes such a polynomial sequence.

Due to the property of orthogonality of such a polynomial sequence, it is 
well-known the following relation between the recurrence coefficients:
\begin{equation} \label{gamrel}
\gamma_n = \alpha_{n-1}\frac{d_n^2}{d_{n-1}^2},\quad n=1, 2, ...,
\end{equation}
where $d_n^2$ is the squared norm of $p_n$.

Such polynomial sequence is said to be classical if it is orthogonal with 
respect to a linear functional ${\bf u}: \mathbb P\to \mathbb C$ which 
fulfills the Pearson-type equation 
\begin{equation} \label{peq} 
{\mathscr D}\left(\phi(x){\bf u}\right)=\psi(x){\bf u},
\end{equation} 
where $\phi$ is a polynomial of degree at most 2, $\psi$ is a polynomial 
of degree $1$, and ${\mathscr D}$ is the differential operator in the 
continuous case, the forward $(\Delta)$ or backward $(\nabla)$ difference 
operator in the discrete case, and the Hahn operator $({\mathscr D}_q)$ 
in the $q$-discrete case.
\begin{rem}
Note that if a linear functional ${\bf u}$ fulfills \eqref{peq} then it also 
fulfills the Pearson-type equation \cite{mr2741220}
\[
{\mathscr D}^*\left(\phi^*(x){\bf u}\right)=\psi^*(x){\bf u},
\]
where\begin{itemize}
\item In the continuous case 
${\mathscr D}^*={\mathscr D}$, so $\phi^*(x)=\phi(x)$ 
and $\psi^*(x)=\psi(x)$.
\item In the discrete case 
$\phi^*(x)=\phi(x)+\psi(x)$, $\psi^*(x)=\psi(x)$, $\Delta^*=\nabla$, 
$\nabla^*=\Delta$. 
\item In the $q$-discrete case, $\phi^*(x)=\phi(x)+(q-1)x\,\psi(x)$, 
$\psi^*(x)=q\psi(x)$, and ${\mathscr D}^*_q={\mathscr D}_{q^{-1}}$.
\end{itemize}
\end{rem}
\begin{defn}\label{def:1.2} 
The polynomial sequence $(p_n(x))_n$ belongs to the {Laguerre constellation} (in short, LC) 
if it is a classical orthogonal polynomial sequence  and $\deg \phi=1$ or $\deg \phi^*=1$.
\end{defn}

The so-called characterization Theorems play a fundamental role, 
i.e., the Theorems which collect those properties that completely define and 
characterize the classical orthogonal polynomials.
One of the many ways to characterize a family of continuous classical 
polynomials (Hermite, Laguerre, Jacobi, and Bessel), which was first posed 
by R. Askey and proved by W. A. Al-Salam 
and T. S. Chihara \cite{mr316772} (see also \cite{mr1273613}), is the {\sf 
structure relation}
\begin{equation} \label{eq:1}
\phi {\mathscr D}p_n(x)=\widetilde a_n p_{n+1}(x)+\widetilde b_n p_{n}(x)+
\widetilde c_n p_{n-1}(x),
\end{equation}
where %${\mathscr D}$ is the differential operator and 
$\widetilde c_n\ne 0$. 

A. G. Garcia et al. proved in \cite{mr1340932} that the relation \eqref{eq:1} 
also characterizes the discrete classical orthogonal polynomials (Hahn, Krawtchouk, 
Meixner, and Charlier polynomials) when the forward replaces the derivative 
difference operator $\Delta$.

Later on, J. C. Medem et al. \cite{mr1850540} characterized the orthogonal 
polynomials which belong to the $q$-Hahn class by a structure relation obtained 
from \eqref{eq:1} replacing the derivative by the $q$-difference operator 
${\mathscr D}_q$ (see also \cite{mr2293537, mr2241592, mr1771452}). One 
of the most general characterization theorem for the $q$-polynomials in the 
$q$-quadratic lattice was done in \cite{MR2653053}.

The structure of this work is the following: in Section \ref{sec:2}, we introduce 
some notations and definitions used throughout the paper. In Section \ref{sec:3} 
we present the main results of the polynomials that belongs to the Laguerre 
constellation and the algebraic properties supporting the presented results. 
In Section \ref{sec:4} we present the specific identities about the different 
families of the Laguerre Constellation.  
Several Wolfram Mathematica${}^\circledR$ codes used to obtain the different 
expressions presented in this work have been included at the end of this work. 
%%%%%%%%%%%%%%%%%%%%%%%%%%%%%%%%%%%%%%%%%%%%%%%%%%%
\section{Preliminaries}\label{sec:2}
%%%%%%%%%%%%%%%%%%%%%%%%%%%%%%%%%%%%%%%%%%%%%%%%%%%
In this section, we will give a brief survey of the operational calculus that 
we will use as well as some basic notations that we will need in the rest of this work.
%%%%%%%%%%%%%%%%%%%%%%%%%%%%%%%%%%%%%%%%%%%%%%%%%%%
\subsection{Basic concepts and results}
%%%%%%%%%%%%%%%%%%%%%%%%%%%%%%%%%%%%%%%%%%%%%%%%%%%
We adopt the following set notations: $\mathbb N_0:=\{0\}\cup\mathbb 
N=\{0, 1, 2, \ldots\}$, and we use the sets $\mathbb Z$, $\mathbb R$, $\mathbb C$, which 
represents the integers, real numbers, and complex numbers, respectively.
Let $\mathbb P$ be the linear space of polynomial functions in $\mathbb C$ 
(in the following, we will refer to them as polynomials) with complex 
coefficients and $\mathbb P^*$ be its algebraic dual space, i.e., 
$\mathbb P^*$ is the linear space of all linear  %applications 
maps ${\bf u}: \mathbb P \to  \mathbb C$. In the following, we will call the 
elements of $\mathbb P^*$  as functionals. 
In general, we will represent the action of a linear functional over a polynomial by
\[
\langle {\bf u}, \pi \rangle,\quad {\bf u}\in \mathbb P^*, \ \pi\in \mathbb P. 
\]
Therefore, a functional is completely determined by a sequence of complex 
numbers $\langle {\bf u}, x^n\rangle=u_n$, $n=0, 1, ...$, the so-called moments 
of the functional.
\begin{defn} Let ${\bf u}\in \mathbb P^*$ be a functional. We say that ${\bf u}$ 
is a {\sf quasi-definite functional} if there exists a polynomial sequence $(p_n)$ 
which is orthogonal with respect to ${\bf u}$, i.e. 
\[
\langle {\bf u}, p_mp_n\rangle = k_n \delta_{mn}, \quad k_n\ne 0, \quad n=0, 1, ...,
\]
where $\delta_{mn}$ is the Kronecker delta.
\end{defn}

In order to obtain our derived identities, we rely on properties of the shifted factorial, 
or Pochhammer symbol, $(a)_n$, and the $q$-shifted factorial, or $q$-Pochhammer 
symbol, $(a; q)_n$. 
For any $n\in \mathbb N_0$, $a$, $q\in \mathbb C$, the shifted factorial is defined as
\[
(a)_n=a(a+1)\cdots (a+n-1),
\]
the $q$-shifted factorial is defined as
\[
(a;q)_n=(1-a)(1-a q)\cdots (1-a q^{n-1}).
\]
We will also use the common notational product conventions
\begin{align*}
(a_1,..., a_k)_n&=(a_1)_n\cdots (a_k)_n,\\
(a_1,..., a_k;q)_n&=(a_1;q)_n\cdots (a_k;q)_n.
\end{align*}

The hypergeometric series ${}_rF_s$,  where $s, r\in\mathbb N_0$, is 
defined by the series %for $z\in \mathbb C$ as %such that %$|z|<1$,  
\cite[(1.4.1)]{mr2656096}
\begin{equation}
\Fhyp{r}{s}{a_1,...,a_r}{b_1,...,b_s}{z}
:=\sum_{k=0}^\infty\frac{(a_1,...,a_r)_k}
{(b_1,...,b_s)_k}\dfrac{z^k}{k!},
\label{eq:2.1}
\end{equation}
and the basic hypergeometric series ${}_r\phi_s$, 
%where $s, r\in\mathbb N_0$,  
is  defined by the series  
%for $q, z\in \mathbb C$ such that $|q|, |z|<1$, 
as \cite[(1.10.1)]{mr2656096}
\begin{equation}
\qhyp{r}{s}{a_1,...,a_r}{b_1,...,b_s}{q,z}:=\sum_{k=0}^\infty
\frac{(a_1,...,a_r;q)_k}{(q,b_1,...,b_s;q)_k}
\left((-1)^kq^{\binom k2}\right)^{1+s-r}z^k.
\label{eq:2.2}
\end{equation}
\noindent Observe that both, hypergeometric and basic hypergeometric 
series, are entire functions of $z$ if $s+1>r$, are convergent for $|z|<1$ 
for $s+1=r$, and divergent if $s+1<r$ and $z\ne 0$, {unless it  terminates}.

Note that when we refer to a hypergeometric 
or basic hypergeometric function with {\it arbitrary 
argument} $z$, we mean that the argument does 
not necessarily depend on the other parameters, 
namely the $a_j$'s, $b_j$'s. However, for the arbitrary 
argument $z$, it very-well may be that the domain 
of the argument is restricted, such as for $|z|<1$.

The next theorem \cite[p. 8]{mr0481884} is useful if one works 
with linear functionals.
\begin{thm}
Let ${\bf u}\in \mathbb P^*$ be a linear functional with moments $(u_n)$. 
Then, ${\bf u}$ is quasi-definite if and only if the Hankel determinants 
$H_n:=\det\left(u_{i+j}\right)_{i,j=0}^n\ne 0$, $n=0, 1, ....$
\end{thm}
%%%%%%%%%%%%%%%%%%%%%%%%%%%%%%%%%%%%%%%%%%%%%%%%%%
\subsection{Definition of the operators in $\mathbb P$ and $\mathbb P^*$}
%%%%%%%%%%%%%%%%%%%%%%%%%%%%%%%%%%%%%%%%%%%%%%%%%%
Next, we will define the backward and forward difference operators as well 
the so called $q$-derivative operator, or Hahn operator.
\begin{eqnarray*}
\Delta f(x)&=&f(x+1)-f(x),\\
\nabla f(x)&=&f(x)-f(x-1),\\
{\mathscr D}_q f(x)&=&\dfrac{f(qx)-f(x)}{x(q-1)},\quad q\ne 1,\ x\ne 0.
\end{eqnarray*}
Since the polynomial sequences depend on $n$, $x$ and their %its 
parameters and we are going to focus more on the variable $n$ along 
this work ($n$ as a discrete variable) 
than in the variable $x$, we need to consider along the paper as well
the operators $\Delta_n$ and $\nabla\!_n$ that are defined 
analogously as the operators $\Delta$ and $\nabla$, i.e. 
\[
\dln f(n;x)=f(n,x)-f(n-1,x),\qquad \Delta_n f(n;x)=f(n+1,x)-f(n,x).
\]
\begin{defn} 
Let ${\bf u}\in \mathbb P^*$ and $\pi \in \mathbb P$, let $\Delta {\bf u}$, 
$\nabla {\bf u}$ and ${\mathscr D}^*_q {\bf u}$ be the linear functionals defined by
\begin{align}
\left\langle \frac d{dx}{\bf u}, \pi\right\rangle =&- \langle {\bf u}, \pi'\rangle,\label{eq:difdef}\\
\langle \Delta {\bf u}, \pi\rangle =&- \langle {\bf u}, \nabla \pi\rangle,\label{eq:backdef}\\
\langle \nabla {\bf u}, \pi\rangle =&- \langle {\bf u}, \Delta \pi\rangle,\label{eq:forwdef}\\
\langle {\mathscr D}^*_q {\bf u}, \pi\rangle =&- q\langle {\bf u}, {\mathscr D}_q \pi\rangle.
\label{eq:qdiffdef}
\end{align}
\end{defn}
Notice that we use the same notation for the operators on $\mathbb P$ and $\mathbb P^*$.
Whenever it is not specified the linear space where an operator acts, it will be understood 
that it acts on the polynomial space $\mathbb P$.
%%%%%%%%%%%%%%%%%%%%%%%%%%%%%%%%%%%%%%%%%%%%%%%%%%%
\section{The Laguerre constellation}\label{sec:3}
%%%%%%%%%%%%%%%%%%%%%%%%%%%%%%%%%%%%%%%%%%%%%%%%%%%
In this section, we are going to present several identities that 
the polynomials belonging to LC fulfill.
First, we are going to show some theoretical aspects and results related to them.
\begin{lem}\label{lem:3.1}
Let ${\bf u}\in \mathbb P^*$ be a quasi-definite classical functional, let $(p_n)$ be the 
polynomial sequence orthogonal with respect to ${\bf u}$. If $(p_n)$ belong to the  LC
then, there exists a numerical sequence $(\lambda_n)_n$ so that 
$(\lambda_n\, p_n)_n$ fulfills the recurrence relation \eqref{ttrr} for which $\gamma_0=0$ and
$\alpha_n+\beta_n+\gamma_n={\sf c}$ for all $n\in \mathbb N$, 
where ${\sf c}$ is a root of $\phi(x)$, or $\phi^*(x)$.
\end{lem}
\begin{proof}
To prove this result it is enough to consider all the families of the Hypergeometric 
orthogonal polynomials scheme and the basic Hypergeometric orthogonal 
polynomials scheme (see e.g. \cite{mr2656096,mr1850540,mr1421640}) that are 
of Laguerre-type, i.e., $\deg \phi=1$ or $\deg \phi^*=1$. These families are the 
Laguerre {\tt (L)}, Charlier {\tt (C)}, Meixner {\tt (M)}, 
big $q$-Laguerre {\tt (bqL)}, $q$-Meixner {\tt (qM)}, %continuous $q$-Laguerre {\tt (cqL)}, 
little $q$-Laguerre {\tt (lqL)}, $q$-Laguerre {\tt (qL)}, $q$-Charlier {\tt (qC)}, 
Stieltjes-Wigert {\tt (SW)},  and the $0$-Laguerre/Bessel {\tt (0LB)} polynomials

Once we obtain all the families belonging to LC it is enough to verify that the conditions 
established for each of the families (see Table \ref{table:1}) are satisfied. 
Observe that the value of $p_n({\sf c})$, where {\sf c} is a zero of $\phi$ or $\phi^*$, 
is known (see \cite{MR2217787, mr2324843}), and these values are non-zero so 
one can define $\lambda_n$ as $1/p_n({\sf c})$. 
\end{proof} 
\begin{rem}Note that Lemma \ref{lem:3.1} is \underline{not} a Characterization 
Theorem since other families of the basic Hypergeometric orthogonal polynomials 
scheme fulfill the condition about the recurrence coefficients, in short (RC), i.e. 
$\alpha_n+\beta_n+\gamma_n$ is a constant, however they do not belong to the LC, 
for example, the Askey-Wilson polynomials
(see \cite[(14.1.4)]{mr2656096}).
\end{rem}

In the next result we write the recurrence relation as a structure-type relation on $n$.
\begin{lem} \label{lem:3.3}
Let $(p_n)_n$ be a polynomial sequence that belongs to the Laguerre constellation. 
For any $x\in \mathbb C$ the recurrence relation \eqref{ttrr} can be written as:
\begin{equation} \label{eq:sr1} 
\phi(x)p_n(x)=\alpha_n \Delta_n p_n(x)-
\gamma_n \Delta_n p_{n-1}(x),
\end{equation}
if $\deg \phi=1$; and
\[
\phi^*(x)p_n(x)=\alpha_n \Delta_n p_n(x)-
\gamma_n \Delta_n p_{n-1}(x),
\]
if $\deg \phi^*=1$.
\end{lem}
\begin{proof} 
Let us assume that $\deg \phi=1$, i.e., $\phi(x)=x-{\sf c}$, then by Lemma \ref{lem:3.1} 
we have that the coefficients of the recurrence relation \eqref{ttrr} fulfill 
$\alpha_n+\beta_n+\gamma_n={\sf c}$, then rewriting the recurrence assuming 
this relation the result holds. The $\deg \phi^*=1$ case is analogous.
\end{proof} 
\begin{rem} 
For the sake of convenience, we are going to replace $x-{\sf c}$ by $\phi(x)$ 
assuming that $\deg \phi=1$. In the case that $\deg \phi\ne 1$ and 
$\deg \phi^*= 1$, then one must replace $\phi$ by $\phi^*$ in the 
further results since such identities and results hold similarly.
 \end{rem}
We write the recurrence relation in the next result as a Sturm-Liouville form 
difference equation on $n$.
\begin{lem}\label{lem:3.4}
Let $(p_n)_n$  be a polynomial sequence that belongs to the Laguerre constellation. 
For any $x\in \mathbb C$ the recurrence relation \eqref{ttrr} can be written as:
\begin{eqnarray} \label{eq:3.1}
\phi(x) p_n(x)&=&d^2_n \dln \dfrac{\gamma_n}{d_n^2}\Delta_n p_n(x),\\[3mm]
&=&d^2_n \Delta_n \dfrac{\alpha_n}{d_n^2}\dln p_n(x).\label{eq:3.2}
\end{eqnarray}
\end{lem}
\begin{proof} 
Starting from \eqref{eq:sr1}, using \eqref{gamrel} and taking into account the 
definition of $\nabla\!_n$ and $\Delta_n$ the results follow.
\end{proof}
\begin{thm}\label{thm:3.1}
Let $(p_n)_n$  be a polynomial sequence that belongs to the Laguerre constellation. 
For any $x\in \mathbb C$, the sequence $(\dln p_n(x))$ is orthogonal with respect to $\alpha_n/d_n^2$, 
and the sequence $(\Delta_n p_n(x))$ is orthogonal with respect to $\gamma_n/d_n^2$.
\end{thm}
\begin{proof} 

Let us fix $x\in \mathbb C$ such that $\omega(x)\ne 0$. 
By Korovkin’s Theorem (see \cite{MR0150565} or 
\cite[Theorem 2.1]{MR4496948})
we have the following closure relation:
\[
\sum_{n=0}^\infty \dfrac{p_n(x)p_n(y)}{d_n^2}=\dfrac 1{\omega(x)}\delta(x-y),
\]
where $\omega$ is the weight function and both the left- and right-hand sides 
should be treated as distributions. From this expression, 
using \eqref{eq:3.2} and Abel's lemma on partial sums 
\cite[p. 313]{ZBMATH2576281} for $y\ne x$ we obtain 
\[
0=\sum_{n=0}^\infty \dfrac{\phi(x)p_n(x)p_n(y)}{d_n^2}=
\sum_{n=0}^\infty \left(\Delta_n\dfrac{\alpha_n}{d_n^2}\dln p_n(x)\right) p_n(y)
=-\sum_{n=0}^\infty \dfrac{\alpha_n}{d_n^2}\dln p_n(x) \dln p_n(y).
\]
Moreover, 
\[
\sum_{n=0}^\infty \dfrac{\alpha_n}{d_n^2}\dln p_n(x) \dln p_n(x)
=\frac{\phi(x)}{\omega(x)}.
\]
The proof for the sequence $(\Delta_n p_n(x))$ is analogous. Hence the result follows.
\end{proof} 
\begin{rem}
Since we have the data for all the families which belong to the LC  it 
is a straightforward calculation to check that $1/d^2_{-1}=0$ 
(since $\gamma_0=0$) and 
\[
\lim_{n\to \infty} \frac{\alpha_n}{d_n^2}=0.
\]
\end{rem} 
\begin{rem}
Observe the previous result does not characterize families belonging 
to LC since, for example, for the Jacobi polynomials the following identity holds:
\[
\dln \left(\dfrac{P_n^{(\alpha,\beta)}(x)}{P_n^{(\alpha,\beta)}(1)}\right)=
\dfrac{P_n^{(\alpha,\beta)}(x)}{P_n^{(\alpha,\beta)}(1)}-
\dfrac{P_{n-1}^{(\alpha,\beta)}(x)}{P_{n-1}^{(\alpha,\beta)}(1)}=
\dfrac{2n+\alpha+\beta}{2(1+a)}(x-1)\dfrac{P_{n-1}^{(\alpha+1,\beta)}(x)}
{P_{n-1}^{(\alpha+1,\beta)}(1)},
\]
and in such a case $\phi(x)=1-x^2$.
\end{rem}
\begin{thm}[Characterization Theorem]\label{thm:3.2}
Let $(p_n)_n$ be an orthogonal polynomial sequence with respect to $\omega$,
such that ${\mathscr D}(\phi\omega)=\psi\omega$%, with $\deg \phi=1$
. 
For any $x\in\mathbb C$ such that $\omega(x)\ne 0$, the following statements 
are equivalent: 
\begin{enumerate}
\item $(p_n)_n$ belongs to the Laguerre constellation.
%\item The polynomials sequence $(\dln p_n(x))$ is orthogonal with  respect to $\alpha_n/d^2_n$.
\item The polynomial sequence $(p_n(x))$ fulfills the second order 
difference equation
\[
\phi(x)p_n(x)=\alpha_n \Delta_n p_n(x)-
\gamma_n \Delta_n p_{n-1}(x),
\]
which is equal to its discrete structure-type relation.
\item The polynomial sequence $(p_n(x))$ satisfies the Sturm-Liouville difference equations 
%\eqref{eq:3.1} and \eqref{eq:3.2}
\[
\phi(x) p_n(x)=d^2_n \dln \dfrac{\gamma_n}{d_n^2}\Delta_n p_n(x)=
d^2_n \Delta_n \dfrac{\alpha_n}{d_n^2}\dln p_n(x).
\]
\end{enumerate}
\end{thm}
\begin{proof}
%Theorem \ref{thm:3.1} implies $1.\rightarrow 2$. 
$1.\leftrightarrow 2.$ since the 2rd item is the recurrence relation 
in such a case. 
$2.\leftrightarrow 3.$  is a consequence of Lemma \ref{lem:3.4}.
\end{proof} 

%%%%%%%%%%%%%%%%%%%%%%%%%%%%%%%%%%%% 
\section{The families}\label{sec:4}
%%%%%%%%%%%%%%%%%%%%%%%%%%%%%%%%%%%% 
Along this section we present several identities related 
to the different families that belong to LC. 
Since there are some relation among them (see Figure 
\ref{fig:1}) we will present such identities only for some 
of the families in order to avoid duplicities. We consider 
the Laguerre, Meixner, Charlier, big $q$-Laguerre, 
little $q$-Laguerre and Stieltjes-Weigert cases.

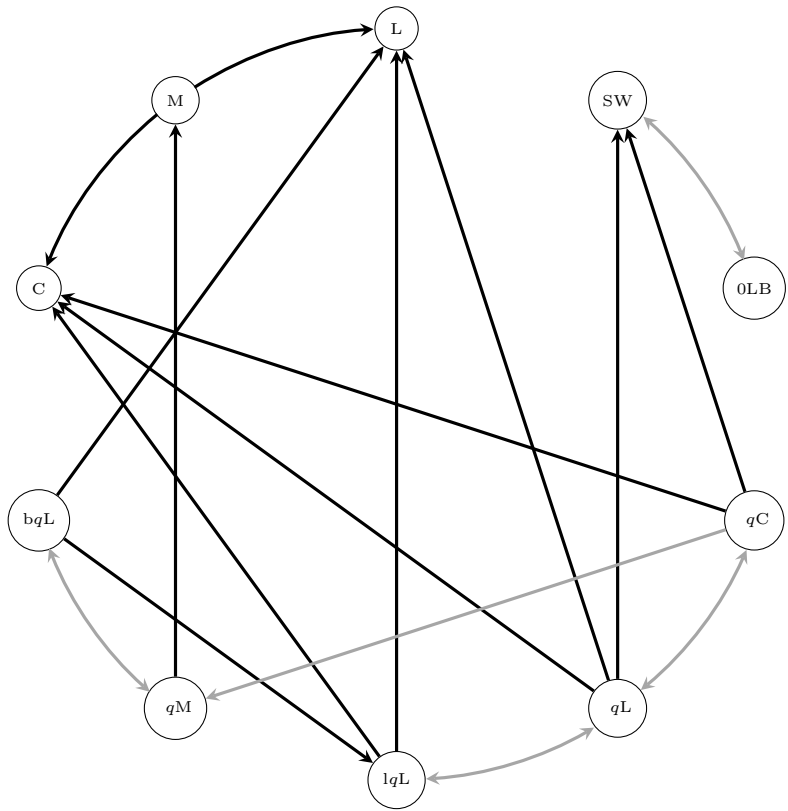
\begin{figure}[!hbt]
\centering
\begin{tikzpicture}
\def \n {10} %% número de vértices
\def \ph {90}. %% angulo girado respecto al eje x
\def \radius {5cm}
\def \margin {3.5} % margen entre línea y círculos
\foreach \s/\nam in {1/L,2/M,3/C,4/b$q$L,5/ $q$M,6/l$q$L,7/ $q$L,8/ $q$C,9/0LB,10/SW}
{
 \node[draw, circle](\s) at ({360/\n * (\s-1)+\ph}:\radius) {\tiny \nam};
% \draw[-stealth, >=latex] ({360/\n * (\s-1)+\ph+\margin}:\radius) 
% arc ({360/\n * (\s-1)+\margin+\ph}:{360/\n * (\s)-\margin+\ph}:\radius);
}
%% limiting cases
\draw [very thick, stealth-] %L <- bqL
(1) -- (4); 
\draw [very thick, -stealth] %L <- lqL
(6) -- (1) ;
\draw [very thick, -stealth] %L <- qL
(7) -- (1);
\draw [very thick, -stealth] %M <- qM
(5) -- (2);
\draw [very thick, -stealth] %C <- lqL
(6) -- (3);
\draw [very thick, -stealth] %C <- qL
(7) -- (3);
\draw [very thick, -stealth] %C <- qC
(8) -- (3);
\draw [very thick, stealth-] %lqL <- bqL
(6) -- (4);
\draw [very thick, -stealth] %SW <- qL
(7) -- (10);
\draw [very thick, -stealth] %SW <- qC
(8) -- (10)
; 
%% particular cases
\draw [very thick, gray!70,-stealth] 
(8) -- (5) %%
;
%% particular cases
\draw[stealth-, very thick, >=latex] ({360/\n * (1-1)+\ph+\margin}:\radius) 
arc ({360/\n * (1-1)+\margin+\ph}:{360/\n * (1)-\margin+\ph}:\radius);
\draw[-stealth, very thick, >=latex] ({360/\n * (2-1)+\ph+\margin}:\radius) 
arc ({360/\n * (2-1)+\margin+\ph}:{360/\n * (2)-\margin+\ph}:\radius);
\draw[stealth-stealth, very thick, >=latex,gray!70] ({360/\n * (4-1)+\ph+\margin+1}:\radius) 
arc ({360/\n * (4-1)+\margin+\ph}:{360/\n * (4)-\margin-2 +\ph}:\radius);
\draw[stealth-stealth, very thick, >=latex,gray!70] ({360/\n * (6-1)+\ph+\margin+1}:\radius) 
arc ({360/\n * (6-1)+\margin+\ph}:{360/\n * (6)-\margin-2 +\ph}:\radius);
\draw[stealth-stealth, very thick, >=latex,gray!70] ({360/\n * (7-1)+\ph+\margin+1}:\radius) 
arc ({360/\n * (7- 1)+\margin+\ph}:{360/\n * (7)-\margin-2 +\ph}:\radius);
\draw[stealth-stealth, very thick, >=latex,gray!70] ({360/\n * (9-1)+\ph+\margin+1}:\radius) 
arc ({360/\n * (9- 1)+\margin+\ph}:{360/\n * (9)-\margin-2 +\ph}:\radius);
\end{tikzpicture}
\caption{Relations between the families in the Laguerre constellation. \label{fig:1} 
The gray lines are the particular cases. 
The black lines are the limiting cases.}
\end{figure}
Before presenting the main results let us show the relations between the families 
we are going to consider with respect to the rest of the families that belong to LC 
(see \cite[p. 20]{Jacob.tPhD}, \cite[remark p. 526]{mr2656096}):
\begin{align}
p_n(x;a,b,1/q)=&\dfrac{1}{(q/b;q)_n}M_n(xq/a;1/a,-b;q),\\
p_n(x;q^\alpha|1/q)=&\dfrac{(q;q)_n}{(q^{\alpha+1};q)_n} L_n^{(\alpha)}(-x;q),\\
L_n^{(\alpha)}(x;q)=&\dfrac{1}{(q;q)_n}C_n(-x;-q^{-\alpha};q),\\
S_n(x/a,1/q)=&{}_2\phi_0\left(q^{-n},0;\--;q,-\frac xa\right)=:l_n(x;a),
\end{align}
where $l_n(x;a)$ denote the 0-Laguerre/Bessel polynomials $(l_n(x;a))$ (see \cite{MR1850541},
\cite[p. 244]{mr1421640}).
%%%%%%%%%%%%%%%%%%%%%%%%%%%%%%%%%%%% 
\subsection{The Laguerre poynomials}
%%%%%%%%%%%%%%%%%%%%%%%%%%%%%%%%%%%%
The Laguerre polynomials can be written in terms of 
hypergeometric series as \cite[\S 9.12]{mr2656096}
\[
L_n^{(\alpha)}(x)=\dfrac{(\alpha+1)_n}{n!}\Fhyp{1}{1}{-n}{\alpha+1}{x}.
\]
In this case $\phi(x)=x$, i.e., this family belongs to the Laguerre constellation  
where the zero of $\phi$ is ${\sf c}=0$, and 
$L_n^{(\alpha)}({\sf c})=(\alpha+1)_n/n!$.
Taking this into account we can state the following result.
\begin{lem} \label{lem:3.5}
For any $\alpha\in \mathbb C$ and any $n\in \mathbb N_0$, 
the following identities hold:
\begin{align}
\dln L_n^{(\alpha)}(x)=&L_n^{(\alpha-1)}(x),\label{lownL}\\
\alpha L_n^{(\alpha)}(x)-(n+1)\Delta_n L_n^{(\alpha)}(x)=&xL_n^{(\alpha+1)}(x),\label{raiupnL}\\
\alpha L_n^{(\alpha)}(x)-(n+\alpha)\dln L_n^{(\alpha)}(x)=&xL_{n-1}^{(\alpha+1)}(x),\label{raidonL}\\
\dlpa{\alpha} L_n^{(\alpha)}(x)=&L_{n-1}^{(\alpha)}(x),\label{lowaL}\\
(n+\alpha-x) L_n^{(\alpha)}(x)-(\alpha+n)\dlpa{\alpha} L_n^{(\alpha)}(x)=&(n+1) L_{n+1}^{(\alpha-1)}(x),\label{raidoaL}\\
(n+1+\alpha-x) L_{n}^{(\alpha)}(x)-x\Delta_\alpha L_{n}^{(\alpha)}(x)=&(n+1) L_{n+1}^{(\alpha)}(x),\label{raiupaL}
\end{align}
where $\Delta_\alpha f(x,\alpha)=f(x,\alpha+1)-f(x,\alpha)$, and $\dlpa{\alpha} f(x,\alpha)=\Delta_\alpha f(x,\alpha-1)$.
\end{lem}
\begin{proof} These identities can be checked by identifying the 
polynomial coefficients on the left and right-hand sides of each identity. 
Hence, the results follow.
\end{proof} 
A direct consequence of the former result is stated as follows.
\begin{thm} \label{thm:3.3}
For any $\alpha\in \mathbb C$, any $x\in \mathbb C$, $x\ne 0$, and any $n\in \mathbb N_0$, 
the polynomial sequence $\big(L_n^{(\alpha+k)}(x)\big)_k$ is orthogonal with respect to 
certain moment functional.
\end{thm}
\begin{proof}
By combining \eqref{lowaL}, \eqref{raidoaL} and \eqref{raiupaL} we have
the following second-order difference equations:
\begin{align}
n L_n^{(\alpha)}(x)=&-(\alpha+n)\nabla\!_\alpha\Delta_\alpha L_n^{(\alpha)}(x)
+(\alpha+n-x)\Delta_\alpha L_n^{(\alpha)}(x)\label{eq:sodeL1}\\
=&-x\Delta_\alpha \dlpa{\alpha} L_n^{(\alpha)}(x)+(\alpha+n-x)\dlpa{\alpha} L_n^{(\alpha)}(x),
\label{eq:sodeL2}
\end{align}
which is connected with the Charlier polynomials case  (see \cite[(9.14.5)]{mr2656096}). 
Using the theory of discrete Sturm-Liouville and 
assuming that for any $n\in \mathbb N_0$, in this case, $\alpha$ is variable, and the 
rest of the elements are fixed, the result holds.
\end{proof} 
\begin{rem} Observe that one can construct, under certain condition on $\alpha$ and $x$,
 certain integral representation for such moment functional.
\end{rem} 
From Lemma \ref{lem:3.5} we can deduce some new identities related to the Laguerre polynomials.
\begin{thm} \label{thm:3.4}
For any $\alpha\in \mathbb C$, and any $n,k \in \mathbb N_0$, 
the following Rodrigues-type identities hold:
\begin{align} \label{eq:LalRF}
L_n^{(\alpha+k)}(x)=&\dfrac{(\alpha+k)_n}{x^k}{(\alpha+k)(\alpha+k-1)}
\dlpa{\alpha} 
%\dfrac{(\alpha+k-1)(\alpha+k-2)}{x}\dlpa{\alpha} 
\cdots {(\alpha+1)\alpha}\dlpa{\alpha} \dfrac{L_{n+k}^{(\alpha)}(x)}{(\alpha)_{n+k}},\\ \label{eq:LnRF}
L_n^{(\alpha+k)}(x)=&\dfrac{(-1)^k (\alpha+1)_{n+k}}{n! x^k}
\nabla^k\hspace{-2mm}_n \dfrac{(n+k)!}{(\alpha+1)_{n+k}}L_{n+k}^{(\alpha)}(x),\\
\label{eq:LalRF3}
L_{n}^{(\alpha)}(x)=&(-1)^k \dfrac{(n+1)_\alpha}{x^\alpha}\Delta_\alpha^k \dfrac{x^\alpha}{(n-k+1)_k}L_{n-k}^{(\alpha)}(x),\\
\label{eq:LnRF2}
L_n^{(\alpha)}(x)=&\Delta_n^k L_{n-k}^{(\alpha+k)}(x),\\
\label{eq:LalRF2}
L_n^{(\alpha+k)}(x)=&\Delta_\alpha^k L_{n+k}^{(\alpha)}(x).
\end{align}
\end{thm}
\begin{proof} 
The first identity holds by mathematical induction on $k$
after a straightforward simplification and using \eqref{raidonL} written in 
the following way:
\[
\dfrac{n!}{(\alpha+k+1)_n} L_n^{(\alpha+k)}(x)=
\dfrac{(\alpha+k)(\alpha+k-1)}{x(n+1)}\dlpa{\alpha} \dfrac{(n+1)!}{(\alpha+k)_{n+1}} L_{n+1}^{(\alpha+k-1)}(x).
\]
The second identity holds by mathematical induction on $k$ after a 
straightforward simplification and using \eqref{raidoaL} written in 
the following way:
\[
\dfrac{n!}{(\alpha+k+1)_n} L_n^{(\alpha+k)}(x)=
-\dfrac{\alpha+k}{x}\dln \dfrac{(n+1)!}{(\alpha+k)_{n+1}} L_{n+1}^{(\alpha+k-1)}(x).
\]
The third identity holds by mathematical induction on $k$ after a 
straightforward simplification and using \eqref{raiupaL} written in 
the following way:
\[
x^\alpha L_n^{(\alpha)}(x)=-(n+1)_\alpha\Delta_n \dfrac{x^\alpha}{(n)_{\alpha}} L_{n-1}^{(\alpha)}(x).
\]

The fourth and the fifth relation hold from \eqref{lownL} and \eqref{lowaL}.
\end{proof}
The last result, but not least interesting, concerning the operators associated 
with the Laguerre polynomials is as follows.
\begin{prop} \label{pro:1}
The Laguerre polynomials fulfill the following identity:
\[\begin{split}
(n+1)L_{n+1}^{(\alpha+1)}(x)=&
(2+a\!-\!x)L_n^{(\alpha)}(x)+(3 x\!-\!4\!-\!2 a)(L_n^{(\alpha)}(x))'
+(2+a\!-\!3 x)(L_n^{(\alpha)}(x))''\\
&+x(L_n^{(\alpha)}(x))'''.
\end{split}\]
\end{prop}
\begin{proof} 
The identity follows using the identities:
\begin{align}
L_n^{(\alpha)}(x)-(L_n^{(\alpha)}(x))'=&L_{n}^{(\alpha+1)}(x),\\
x (L_n^{(\alpha)}(x))'-(x-\alpha)L_n^{(\alpha)}(x)=&(n+1)L_{n+1}^{(\alpha-1)}(x),
\end{align}
applying the first twice, and later the second one once mapping $\alpha\mapsto \alpha+2$.
\end{proof}
%%%%%%%%%%%%%%%%%%%%%%%%%%%%%%%%%%%%
\subsection{The Charlier polynomials}
%%%%%%%%%%%%%%%%%%%%%%%%%%%%%%%%%%%%
The Charlier polynomials can be written in terms of 
hypergeometric series as \cite[\S 9.14]{mr2656096}
\[
C_n(x;a)=\Fhyp20{-n,-x}{\--}{-\frac 1a}.
\]
In this case $\phi(x)=x$ and $\phi^*(x)=a$, i.e., ${\sf c}=0$, and $C_n({\sf c};a)=1$.
Taking this into account we can state the following result.
\begin{lem} \label{lem:3.6}
For any $a\in \mathbb C$ and any $n\in \mathbb N_0$, the following identities hold:
\begin{align}
a\Delta_nC_n(x;a)=&-{x}\,C_{n}(x-1;a),\label{rainC}\\
n C_n(x;a)+a (C_n(x;a))'=&(n+1)C_{n-1}(x;a),\label{lowDC}\\
(a-n) C_n(x;a)+n\dln C_{n}(x;a)=&a\, C_{n}(x+1;a),\label{lowdonC}\\
(a-n-1) C_n(x;a)+a \Delta_n C_{n}(x;a)=&a\, C_{n+1}(x+1;a).\label{lowupnC}
\end{align}
\end{lem}
\begin{proof} 
\eqref{lowDC} is a direct consequence of the identity presented in 
Remark in \cite[p. 249]{mr2656096}
\[
\dfrac{(-a)^n}{n!}C_n(x;a)=L_n^{(x-n)}(a).
\]
The other identities can be checked by identifying the polynomial coefficients 
on the left and right-hand sides of each identity. Hence, the results follow
\end{proof}
From Lemma \ref{lem:3.6} we can deduce some new identities related to the Charlier polynomials.
\begin{thm} \label{thm:3.5}
For any $a\in \mathbb C$, and any $n,k \in \mathbb N_0$, 
the following Rodrigues-type identities hold:
\begin{align} 
C_n(x;a)=&\dfrac{(-a)^k}{(x+1)_k}\Delta^k\hspace{-2mm}_n C_n(x+k;a), \label{eq:CnRF}\\
C_n(x;a)=&\dfrac{n!}{a^n}\Delta^k\hspace{-2mm}_n \dfrac{a^{n-k}}{(n-k)!}C_{n-k}(x-k;a). \label{eq:CnRF2}
\end{align}
\end{thm}
\begin{proof} 
The first identity is a direct consequence of \eqref{rainC}; and the second is due the identity:
\[
(n+1)!\Delta_n \dfrac{a^n}{n!} C_n(x;a)=a^{n+1} C_{n+1}(x+1;a).
\] 
\end{proof}
%%%%%%%%%%%%%%%%%%%%%%%%%%%%%%%%%%%%
\subsection{The Meixner polynomials}
%%%%%%%%%%%%%%%%%%%%%%%%%%%%%%%%%%%%
The Meixner polynomials can be written in terms of 
hypergeometric series as \cite[\S 9.10]{mr2656096}
\[
M_n(x;\beta,c)=\Fhyp{2}{1}{-n,-x}{\beta}{1-\frac 1c}.
\]
In this case $\phi(x)=x$ and $\phi^*(x)=c(x+\beta)$, therefore we must consider 
two cases, ${\sf c}_1=0$, and ${\sf c}_2=-\beta$, for which we have
$M_n({\sf c}_1;\beta,c)=1$, and $M_n({\sf c}_2;\beta,c)=1/{c^n}$.
Taking into account Lemma \ref{lem:3.1} we can state the next result.
\begin{lem} \label{lem:3.7}
The polynomial sequence $(M_n(x;\beta,c)/M_n({\sf c}_2;\beta,c))_n$ fulfills 
the recurrence relation 
\[
(x+\beta)p_n(x)=\alpha^{\tt M}_n p_{n+1}(x)
-(\alpha^{\tt M}_n+\gamma^{\tt M}_n) p_{n}(x)
+\gamma^{\tt M}_np_{n-1}(x),
\]
with initial condition $p_0(x)=1$, 
\[
\alpha^{\tt M}_n=\dfrac{n+\beta}{c-1},\qquad 
\gamma^{\tt M}_n=\dfrac{n c}{c-1},
\]
and fulfills the second order difference equation
\[
(x+\beta)p_n(x)=d_n^2 \Delta_n \dfrac{\gamma^{\tt M}_n}{d_n^2} \dln p_n(x)=
d_n^2 \nabla\!_n\dfrac{\alpha^{\tt M}_n}{d_n^2}\Delta_n p_n(x)=
\gamma^{\tt M}_n \nabla\!_n\Delta_n p_n(x)+(\alpha^{\tt M}_n-\gamma^{\tt M}_n)\Delta_n p_n(x).
\]
\end{lem}
As in the Laguerre polynomial case, the next result follows.
\begin{lem} \label{lem:3.8}
For any $\beta, c\in \mathbb C$, $\beta \not\in\{ 0,2\}$, 
$c\not\in\{0, 1\}$, and any $n\in \mathbb N_0$, 
the following identities hold:
\begin{align}
\hspace{-7mm}\dfrac{\beta c}{c-1}\,\Delta_nM_n(x;\beta,c)=&{x} \,M_{n}(x-1;\beta+1,c),\label{rainM} \\
\hspace{-7mm}M_n(x;\beta,c)+\frac c{c-1}\Delta_nM_n(x;\beta,c)=&\frac{x+\beta}{\beta}\,M_{n}(x;\beta+1,c),\label{raiupnM}\\
\hspace{-7mm}M_n(x;\beta,c)+\frac 1{c-1}\dln M_n(x;\beta,c)=&\frac{x+\beta}{\beta}\,M_{n-1}(x;\beta+1,c),\label{raidonM}\\
\hspace{-7mm}\dfrac{c\beta(1-\beta)}{c-1}\,\nabla\!_\beta M_n(x;\beta,c)=&\, x n M_{n-1}(x-1;\beta+1,c),\label{lowbM}\\
\nonumber\hspace{-7mm}\dfrac{\beta(\beta-1) c}{(\beta+n)(c-1)} M_{n+1}(x+1;\beta-1,c)=&\left(x +\beta+ \dfrac{\beta(\beta-1)}{(\beta+n)(c-1)}\right) M_n(x;\beta,c)\\
& +{(x +\beta)}\Delta_\beta M_n(x;\beta,c), \label{raiupbM}\\
\nonumber\hspace{-7mm}\dfrac{(\beta-1)(\beta-2) c}{(\beta-1+n)(c-1)} M_{n+1}(x+1;\beta-2,c)=&\left(x +\beta-1+ \dfrac{(\beta-1)(\beta-2)}{(\beta-1+n)(c-1)}\right) M_n(x;\beta,c) \\&-{\dfrac{(\beta-1)(\beta-2)}{(\beta-1+n)(c-1)}}\nabla_\beta M_n(x;\beta,c),\label{raidobM}\\
\dfrac {\partial}{\partial c}c^n (\beta)_n M_n(x;\beta,c)=& 
\dfrac{n (x+\beta)}{c+n+\beta} c^ n (\beta+1)_n M_{n-1}(x;\beta+1,c),\label{lowcM}\\
\nonumber c (1-\beta) c^n (\beta)_n \beta M_{n+1}(x;\beta-1,c) =&c (1-c) \dfrac {\partial}{\partial c} 
c^n (\beta)_n M_n(x;\beta,c) \\&- \left((c-1) x+n-(n+1) c+c \beta\right) 
c^ n (\beta)_n \beta M_n(x;\beta,c).\label{raidcM}
\end{align}
%where $\Delta_\beta f(x,\beta)=f(x,\beta+1)-f(x,\beta)$, and $\nabla\!_\beta f(x,\beta)=\Delta_\beta f(x,\beta-1)$.
\end{lem}
\begin{proof} These identities can be checked by identifying the 
polynomial coefficients on the left and right-hand sides of each identity and with the help of 
 Wolfram Mathematica 13${}^\circledR$.
\end{proof}
A direct consequence is the fact that this polynomial sequence is orthogonal with respect to the 
parameter $\beta$.
\begin{thm} \label{thm:3.6} For any $\beta, c\in \mathbb C$, with $-\beta \not\in \mathbb N_0$,  
any $x\in \mathbb C$,  $x\not\in\{0,-\beta\}$, and any $n\in \mathbb N_0$, the polynomial sequence 
$\big(M_n(x;\beta+k,c)\big)_k$ 
%and $\big({\partial_c}\, c^n (\beta)_n M_n(x;\beta,c)\big)_n$ are 
is orthogonal with respect to certain moment functional.
\end{thm}
\begin{proof}
By combining \eqref{lowbM}, \eqref{raiupbM} and \eqref{raidobM} we have
the following second order difference equations:
\begin{equation} \label{eq:sodeM}
\hspace*{-1.5mm}n x M_{n}(x;\beta,c)\!=\!\dfrac{\beta(\beta\!-\!1)}{c-1}\dlpa{\beta}\Delta_\beta M_{n}(x;\beta,c)
-\left(\!\dfrac{\beta(\beta\!-\!1)}{c-1}\!+\!(\beta+x)(\beta+n)\!\right)\!\Delta_\beta M_{n}(x;\beta,c),\!\!
\end{equation}
which is connected with the Continuous Hahn polynomials case 
(see \cite[(9.4.5)]{mr2656096}). Using the theory of discrete  Sturm-Liouville and 
assuming that for any $n\in \mathbb N_0$, in this case, $\beta$ is variable, and the rest of the 
elements are fixed, the result holds.
%For the second polynomial sequence it enough to take into account the relation 
%\eqref{lowcM}. However, due the following identity:
%\begin{eqnarray} \nonumber
%c (c-1) \dfrac{\partial^2}{\partial c^2} c^n (\beta)_n M_n(x;\beta,c) 
%+ \left((c-1) x+(c \beta-(m-1) c+(m-1))\right) 
%\dfrac{\partial}{\partial c} c^n (\beta)_n M_n(x;\beta,c) \\
%=(n x+n \beta] )c^n (\beta)_n M_n(x;\beta,c) ,\label{eq:sodeM2}
%\end{eqnarray}
%which is connected with the non-symmetric Jacobi polynomials case 
%(see \cite[(9.8.6)]{mr2656096}) and by using the theory of Sturm-Liouville, 
%the result holds.
\end{proof} 
\begin{rem} \begin{itemize} 
\item Note that in \cite{MR2494714} the authors extended the 
orthogonality relations for the Meixner polynomials through 
the orthogonality relations of the 
Continuous Hahn Polynomials.
\item Observe that one can construct, under certain condition on the parameters, $n$ and $x$,
 certain integral representation for such moment functional.
 \end{itemize}
\end{rem}
%%%%%
From Lemma \ref{lem:3.8} we can deduce some new identities related to the Meixner polynomials.
\begin{thm} \label{thm:3.6}
For any $\beta, c\in \mathbb C$, and any $n,k \in \mathbb N_0$, 
the following Rodrigues-type identities hold:
\begin{align} \label{eq:MnRF}
M_{n}(x;\beta+k,c)=&\dfrac{(\beta)_k c^k}{(c-1)^k (x+1)_k}
\nabla^k\hspace{-2mm}_n M_{n+k}(x;\beta,c),\\ 
M_{n}(x;\beta+k,c)=&\dfrac{(\beta)_k}{(1-c)^k (x+\beta)_k(-c)^n}
\Delta^k\hspace{-2mm}_n (-c)^{n+k}M_{n+k}(x;\beta,c), \label{eq:MnRF2}\\
\label{eq:MbRF}
M_n(x;\beta+k,c)=&\dfrac{c^k (\beta+k-1)(2-k-\beta)}{(c-1)^k (x+1)_k (n+1)_k}
\dlpa{\beta} \cdots \beta(1-\beta)\dlpa{\beta} M_{n+k}(x+k;\beta,c),\\
\label{eq:LalRF2}
c^n (\beta+k)_n M_n(x;\beta+k,c)=&\dfrac{(c+n+\beta+k)}{(n+1)_k (x+\beta)_k}
\dfrac{\partial}{\partial c}\cdots (c+n+\beta+k)\dfrac{\partial}{\partial c}
c^{n+k} (\beta)_{n+k} M_{n+k}(x;\beta,c).
\end{align}
\end{thm}
\begin{proof} 
The first identity holds by mathematical induction on $k$
after a straightforward simplification and using \eqref{rainM}.
The second identity holds by mathematical induction on $k$ after a 
straightforward simplification and using \eqref{raidonM} written in 
the following way:
\[
\Delta_n{(-c)^n}M_n(x;\beta,c)
=\dfrac{x+\beta}{\beta}(1-c)(-c)^n M_{n-1}(x;\beta+1,c).
\]
The third identity holds by mathematical induction on $k$ after a 
straightforward simplification and using \eqref{lowbM}.
The fourth and the fifth relation hold from \eqref{lowcM}.
\end{proof}
%%%%%
The last result concerning the operators associated with the Meixner 
polynomials is as follows.
\begin{prop} \label{pro:2}
The following identity holds:
\[\begin{split}
\dfrac{(\beta c+c x+c-x)}{c} M_n(\beta ,c,x)+&\dfrac{\left(2 \beta c^2+2 c^2 x+2 c^2-c x-x\right)}{(c-1) c}
\nabla M_n(\beta ,c,x)\\
+\dfrac{c (-2 \beta +3 \beta c+3 c x+3 c-3 x-2)}{(c-1)^2} \nabla^2 M_n(\beta ,c,x)) +&\dfrac{c^2 (\beta +x+1) }{(c-1)^2}\nabla^3 M_n(\beta ,c,x)
\\ =&\dfrac{(\beta +n) (\beta +n+1) }{\beta }M_{n+1}(\beta +1,c,x).
\end{split}
\]
\end{prop}
\begin{proof}
The proof follows after using an algorithm written in  Wolfram Mathematica 13${}^\circledR$.
\end{proof}
%%%%%%%%%%%%%%%%%%%%%%%%%%%%%%%%%%%%
\subsection{The big $q$-Laguerre polynomials}
%%%%%%%%%%%%%%%%%%%%%%%%%%%%%%%%%%%%
The big $q$-Laguerre polynomials can be written in terms of 
basic hypergeometric series as \cite[\S 14.11]{mr2656096}
\[%\begin{equation} \label{bql-def}
p_n(x;a,b;q)=\qhyp32{q^{-n},0,x}{aq,bq}{q,q}.
\]%\end{equation} 
In this case $\phi(x)=(x-aq)(x-bq)$ and $\phi^*(x)=abq(1-x)$, i.e., 
${\sf c}=1$, and $p_n({\sf c};a,b;q)=1$. Observe these polynomials 
are symmetric in the parameters $a$ and $b$. Taking this into account 
we can state the following result.
\begin{lem} \label{lem:4.5}
For any $a, b\in \mathbb C$, and any $n\in \mathbb N_0$, 
the following identities hold:
\begin{align}
\hspace{-7mm}{\phi(1)q^{n}}\,\Delta_n p_n(x;a,b;q)=&(x-1) \,p_n(xq;aq,bq;q),\label{rainbqL} \\
\hspace{-7mm}(-1+a) (-1+b)q^{n+1}\,p_{n+1}(x/q;a/q,b/q;q)=&
(a b q^{n+1}\!-\! a q^{n+1}\!-\! bq^{n+1}+1)p_n(x;a,b;q)\nonumber \\
&\hspace{-1mm}+\alpha_n^{\tt bqL}\Delta{}_n p_n(x;a,b;q),\label{raiupnbqL}\\
\hspace{-7mm} (-1+a) (-1+b)q^{n+1}\,p_n(x/q;a/q,b/q;q)=&
q(a b q^{n}\!-\! a q^{n}\!-\! bq^{n}+1)p_n(x;a,b;q) \nonumber\\
&+q^n \gamma_n^{\tt bqL} \dln p_n(x;a,b;q),\label{raidonbqL}\\
\hspace{-7mm}{(a-1)\phi(1)}\,{\mathscr D}\!_{q,a}\, p_n(x;a,b;q) =&\dfrac{a (1-q^n)}{q^{n-1}}(x-1)p_{n-1}(xq;aq,bq;q),\label{lowabqL}\\
\nonumber\hspace{-7mm}
(a-1) b q^{n+1}(a q-1) p_n(x;a/q,b;q)= (a q^{n+1}-1) (x-a q) &{\mathscr D}\!_{q,a}\, p_n(x;a,b;q)\\+((1-a q) (x-(a-b+a b) q^{n+1})&+ (a q^{n+1}-1) (x-a q))p_n(x;a,b;q), \label{raiupbbqL}\\
\nonumber\hspace{-7mm}(a-1) (a q-1) (b q^{n+1}-1) \, p_{n+1}(x;a/q,b;q)=&
(1-a) b (a q-1) q^{n+1}{\mathscr D}\!_{1/q,a} p_n(x;a,b;q)\\ &+(a q-1) (x+a b
 q^{n+1}-a-b q^{n+1}) p_n(x;a,b;q)\label{raidobbqL}
\end{align}
where $\alpha_n^{\tt bqL}$ and $\gamma_n^{\tt bqL}$ are the big $q$-Laguerre 
recurrence relation coefficients (see Table \ref{table:1}), and 
\[{\mathscr D}\!_{q,a} f(x,a)=\dfrac{f(x,q a)-f(x,a)}
{a(q-1)}.
\]
\end{lem}
\begin{proof} These identities can be checked by identifying the polynomial 
coefficients on the left and right-hand sides of each identity and with the help of 
 Wolfram Mathematica 13${}^\circledR$.
\end{proof}
A direct consequence is the fact that this polynomial sequence is orthogonal with respect to the 
parameters $a$ and $b$.
\begin{thm} \label{thm:4.6} For any $a, b\in \mathbb C$, any $x\in \mathbb C$, 
$x\not\in\{1, a q, bq\}$, and any $n\in \mathbb N_0$, the polynomial sequences 
$\big(p_n(x;aq^k,b;q)\big)_k$ and 
$\big(p_n(x;a,bq^k;q)\big)_k$
are orthogonal with respect to certain moment functional.
\end{thm}
\begin{proof}
By combining \eqref{lowabqL}, \eqref{raiupbbqL} and \eqref{raidobbqL} we have
the following second order difference equations:
\begin{align} \label{eq:sodeM1}
(x-1)a(1-q^n)p_n(x;a,b;q)=& {\sf t}_n(a) {\mathscr D}_{1/q,a} 
{\mathscr D}_{q,a} p_n(x;a,b;q) 
-({\sf t}_n(a)-{\sf t}^*_n(a) ) {\mathscr D}_{q,a} p_n(x;a,b;q),\\
\label{eq:sodeM2}
=& {\sf t}^*_n(a) {\mathscr D}_{q,a} {\mathscr D}_{1/q,a} p_n(x;a,b;q) 
-({\sf t}_n(a) -{\sf t}^*_n(a) ) {\mathscr D}_{1/q,a} p_n(x;a,b;q),
\end{align}
where 
${\sf t}_n(a)=(a q^{n+1}-1) (x-a q)$, and 
${\sf t}^*_n(a)=(a-1) b q^{n+1}(a q-1)$,
which is connected with the big $q$-Jacobi polynomials case 
(see \cite[(14.5.5)]{mr2656096}). 
By using the theory of Sturm-Liouville, the result holds. The result 
also holds for the parameter $b$ by symmetry.
\end{proof} 
From Lemma \ref{lem:4.5} we can deduce some new identities related 
to the big $q$-Laguerre polynomials.
\begin{thm} \label{thm:3.6}
For any $a, b\in \mathbb C$, and any $n, k \in \mathbb N_0$, 
the following Rodrigues-type identities hold:
\begin{align} \label{eq:bqLnRF}
p_n(q^kx;aq^k,bq^k;q)\!=&\dfrac{(aq,bq;q)_k}{ (x;q)_k}
\nabla^k\hspace{-2mm}_n p_{n+k}(x;a,b;q),\\ \label{eq:bqLaRF}
p_n(q^kx;aq^k,bq^k;q)\!=&\dfrac{q^{nk}q^{{k+1}\choose 2} (bq;q)_k}{(q^{n+1},x;q)_k}
\dfrac{(a q^{k\!-\!1}\!-\!1)(1\!-\!a q^k)}{aq^k}{\mathscr D}_{q,a} \cdots 
\dfrac{(a\!-\!1)(1\!-\!a q)}{aq}{\mathscr D}_{q,a} p_{n+k}(x;a,b;q). 
\end{align}
\end{thm}
\begin{proof} 
The identities holds by mathematical induction on $k$
after a straightforward simplification and using \eqref{rainbqL} and 
\eqref{lowabqL}.
\end{proof}
The last result concerning the operators associated with the big $q$-Laguerre
polynomials is as follows.
\begin{prop} \label{pro:3}
The following identity holds:
\[
\frac{\left(c q^{n+1}-1\right) \left(c q^{n+2}-1\right) \left(\beta q^{n+1}-1\right) 
\left(\beta q^{n+2}-1\right)}{q^{4 n} (c q-1) (\beta q-1)} p_{n+1}(q^4 x, a q, b q;q)
=\sum_{k=0}^4 p_k(x){\mathscr D}^k_{q}p_{n}(x;a,b;q),
\]
where the polynomial coefficients are 
\begin{align*}
p_0(x)=&\frac{1}{q}\big((q^4 x\!-\!1)(a+b+a q+b q+a b q)
+(a+b) (1+q) (aq^2\!-\!1) (bq^2\!-\!1)\big),\\
p_1(x)=&\dfrac{(q\!-\!1)}{q^2}\big((q^4 x-1) (a b+a^2 q+2 a b q+a^2 b q+b^2 q+a b^2 q+
a b q^2+a^2 b q^2+a b^2 q^2)\\
&+(aq^2\!-\!1) (bq^2\!-\!1) (a b+a^2 q+2 a b q+b^2 q+a b q^2)\big),\\
p_2(x)=&\frac{a b (q\!-\!1)^2}{q^2} 
\big((q^4x\!-\!1) (a+b+a b+a q+a^2 q+b q+2 a b q+b^2 q+a b q^2)\\
&+(a+b) (1+q) (a q^2\!-\!1) (b q^2\!-\!1)\big),\\
p_3(x)=&\frac{a^2 b^2 (q\!-\!1)^3}{q^2} 
\big((q^4x\!-\!1)(1+a+b+a q+b q)+(a q^2\!-\!1) (b q^2\!-\!1)\big),\\
p_4(x)=&\dfrac{a^3 b^3 (q\!-\!1)^4}{q^2} (q^4 x\!-\!1).
\end{align*}
\end{prop}
\begin{proof}
The proof follows after using an algorithm written in Wolfram Mathematica 13${}^\circledR$.
\end{proof}
%%%%%%%%%%%%%%%%%%%%%%%%%%%%%%%%%%%%
\subsection{The little $q$-Laguerre/Wall polynomials}
%%%%%%%%%%%%%%%%%%%%%%%%%%%%%%%%%%%%
The little $q$-Laguerre polynomials can be written in terms of 
basic hypergeometric series as \cite[\S 14.20]{mr2656096}
\[%\begin{equation} \label{lql-def}
p_n(x;a|q)=\qhyp21{q^{-n},0}{aq}{q,qx}.
\]%\end{equation} 
In this case $\phi(x)=(1-x)x$ and $\phi^*(x)=ax$, i.e., 
${\sf c}=0$, and $p_n({\sf c};a|q)=1$.
 Taking this into account we can state 
the following result.
\begin{lem} \label{lem:4.6}
For any $a\in \mathbb C$, and any $n\in \mathbb N_0$, 
the following identities hold:
\begin{align} \label{lownlql}
\hspace{-8mm}
\frac{q^{1-n}\,x}{aq\!-\!1} p_{n\!-\!1}(x;aq|q)=&\dln p_n(x;a|q),\\
\label{lownuplql}
\hspace{-8mm}
 (1\!-\!a)(x\!-\!1) p_{n}(x/q;a/q|q)=&(a q^{n+1}\!-\!1) \Delta_n p_{n}(x;a|q)\!-\!(a x+1 \!-\!a) p_{n}(x;a|q),\\
 \label{lowndolql}
\hspace{-8mm}
 (1 \!-\! a) (x\!-\!1) p_{n\!-\!1}(x/q;a/q|q)=& (a x + (q^n \!-\! 1) a) \dln p_{n}(x;a|q)\!-\! (a x +1\!-\! a) p_{n}(x;a|q),\\
 \hspace{-8mm} \dfrac{a (1-q^n)x}{q^{n-1}(1-a)(1-aq)} p_{n-1}(x;aq|q) = &
  {\mathscr D}\!_{1/q,a}\, p_n(x;a|q),\label{lowalqL}\\
 \hspace{-7mm}\nonumber
 (a-1) q^n (aq-1)p_n(x;a/q|q)=&  (aq^{n+1}-1)x {\mathscr D}\!_{1/q,a}\, p_n(x;a|q)\\
&+(a q (q^n \!-\! 1) x + q^n (a q \!-\! 1) (a \!-\! 1)) p_n(x;a|q),\label{raiadolqL}\\
 \hspace{-7mm}\nonumber
a (a\!-\!1) q^{n\!-\!1} (q^n\!-\!1)p_{n-1}(x;a/q|q)=&  (a-1)(x+aq^{2n-1}-aq^{n-1}) 
{\mathscr D}\!_{q,a}\, p_n(x;a|q)\\
&+(a (q^n \!-\! 1) x + (1 \!-\! a) a q^{n \!-\! 1} (q^n \!-\! 1)) p_n(x;a|q).\label{raiadolqL}
 \end{align}
\end{lem}
\begin{proof} These identities can be checked by identifying the polynomial 
coefficients on the left and right-hand sides of each identity and with the help of 
 Wolfram Mathematica 13${}^\circledR$.
\end{proof}
A direct consequence is the fact that this polynomial sequence is orthogonal 
with respect to the parameter $a$.
\begin{thm} \label{thm:4.8} 
For any $a\in \mathbb C$, any $x\in \mathbb C$, 
$x\not\in\{1, a q\}$, and any $n\in \mathbb N_0$, the polynomial sequence 
$\big(p_n(x;aq^k|q)\big)_k$ is orthogonal with respect to certain moment functional.
\end{thm}
\begin{proof} 
By combining \eqref{lowabqL}, \eqref{raiupbbqL} and \eqref{raidobbqL} we have
the following second-order difference equations:
\begin{align} \label{eq:sodelqL1}
\hspace{-7mm}aqx  (q^n-1)p_n(x;a|q)=&{\sf s}_n(a) {\mathscr D}_{1/q,a} 
{\mathscr D}_{q,a} p_n(x;a|q) 
-({\sf s}_n(a)-{\sf s}^*_n(a) ) {\mathscr D}_{q,a} p_n(x;a|q)\\
\label{eq:sodelqL2}
\hspace{-7mm}=& {\sf s}^*_n(a) {\mathscr D}_{q,a} {\mathscr D}_{1/q,a} p_n(x;a|q)
-({\sf s}_n(a) -{\sf s}^*_n(a) ) {\mathscr D}_{1/q,a}  p_n(x;a|q)
\end{align}
where 
${\sf s}_n(a)=(a-1) q^n (aq-1)$, and 
${\sf s}^*_n(a)=(1-a q^{n+1}) x$,
which is connected with the big $q$-Laguerre polynomials case 
(see \cite[(14.11.5)]{mr2656096}). 
Using the theory of discrete Sturm-Liouville and assuming that for any $n\in \mathbb N_0$, 
in this case, $a$ is variable, and the rest of the elements are fixed, the result holds.
\end{proof} 
The last result concerning the operators associated with the little $q$-Laguerre
polynomials is as follows.
\begin{prop} \label{pro:4}
The following identity holds:
\[
\frac{\left(a q^{n+1}-1\right) \left(a q^{n+2}-1\right)}{q^{3 n} (q-1)(a q-1)} 
p_{n+1}(q^4 x, a q|q)
=\sum_{k=0}^4 {\sf p}_k(x){\mathscr D}^k_{q}p_{n}(x;a|q)
\]
where the polynomial coefficients are 
\begin{align*}
{\sf p}_0(x)=&\frac{1}{q^3}\big(q^4 x(a q^3 - a q-1 - q - q^2)
+(a q^2\!-\!1) (a q^3-1 - q - q^2 )\big),\\
{\sf p}_1(x)=&\dfrac{(q\!-\!1)}{q^5}\big(a q^9 x^2+ q^4 x (a^2 q^4+ a^2 q^3 -1 - a - q - 2 a q - q^2 - a q^2 + a q^3 + a q^4 )\\
&+(a q^2\!-\!1) (a q^4-1 - q - q^2 + a q^3)\big),\\
{\sf p}_2(x)=&\frac{ (q\!-\!1)^2}{q^6} 
\big(a q^8x^2 (1 + q + a q)+ q^3x
(a^2 q^5+ 2 a^2 q^4-a - q - 2 a q - 2 a q^2 + a^2 q^3 + a q^4)\\
&+(a q^2\!-\!1)(a q^3\!-\!1)\big),\\
{\sf p}_3(x)=&\frac{a (q\!-\!1)^3 x}{q^4} 
\big(q^5x  (1 + a + a q)+ (1 + q) (-1 + a q^3)\big),\\
{\sf p}_4(x)=&{a^2 (q\!-\!1)^4} x^2.
\end{align*}
\end{prop}
\begin{proof}
The proof follows after using an algorithm written in  Wolfram Mathematica 13${}^\circledR$.
\end{proof}
%%%%%%%%%%%%%%%%%%%%%%%%%%%%%%%%%%%%
\subsection{The Stieltjes-Wigert polynomials}
%%%%%%%%%%%%%%%%%%%%%%%%%%%%%%%%%%%%
These polynomials can be written in terms of basic hypergeometric series 
as \cite[\S 14.21]{mr2656096}
\begin{equation} \label{sw-def}
S_n(x;q)=\qhyp11{q^{-n}}{0}{q,-q^{n+1}x}.
\end{equation} 
In this case $\phi(x)=x^2$ and $\phi^*(x)=x$, i.e., 
${\sf c}=0$, and $S_n({\sf c};q)=1$. Taking this into account we can state 
the following result.
\begin{lem} \label{lem:4.7} 
For any $n\in \mathbb N_0$, the following identities hold:
\begin{align} \label{rainSW}
\dln S_n(x;q)=&-q^n x S_{n-1}(qx;q),\\
\label{lowupnSW}
S_n(x;q)+q^{-n-1} \Delta_n\, S_n(x;q)=&S_{n+1}(x/q;q), \\
\label{lowdwnSW}
S_n(x;q)-(1-q^{-n})\dln S_n(x;q)=&S_{n}(x/q;q).
\end{align}
\end{lem}
From Lemma \ref{lem:4.7} we can deduce some a new identities related 
to the Stieltjes-Wigert polynomials.
\begin{thm} \label{thm:4.9}
For any $n,k \in \mathbb N_0$, the following Rodrigues-type identities hold:
\begin{align} \label{SWnRF}
S_{n}(x;q)=&\dfrac{1}{(-x)^k} {q^{-n}}\dln \cdots 
{q^{-n}}\dln S_{n+k}(x/q^k;q),\\ 
\label{SWnRF2}
S_{n}(x;q)=&{(-1)^n (q;q)_{n}}{q^{-n}}\Delta_n\cdots 
{q^{-n+k}}\Delta_n \dfrac{(-1)^n}{(q;q)_{n-k}}S_{n-k}(xq^k;q).
\end{align}
\end{thm}

\begin{proof} 
The first identity holds by mathematical induction on $k$
after a straightforward simplification and using \eqref{rainSW}.

The second identity holds by mathematical induction on $k$ after a 
straightforward simplification and using \eqref{lowupnSW} written in 
the following way:
\[
\dfrac{(-q)^n}{(q;q)_n} S_{n}(x;q)=\Delta_n \dfrac{(-1)^n}{(q;q)_{n-1}}
S_{n-1}(x q;q).
\]
\end{proof}

The last result concerning the operators associated with the Stieltjes-Wigert polynomials 
is as follows.
\begin{prop} \label{pro:5}
The following identity holds:
\[
(q-x) S_n(x;q)+(1-q)x{\mathscr D}_{1/q} S_n(x;q)=q S_{n+1}(x/q^2;q).
\]
\end{prop}
\begin{proof}
The proof follows after using an algorithm written in Wolfram Mathematica 13${}^\circledR$.
\end{proof}

\begin{table}
\renewcommand*{\arraystretch}{2}
\[
\begin{array}{c@{\hspace{2mm}}cccc}
{\rm } &\alpha_n & \beta_n & \gamma_n & d_n^2 \\ \toprule
{\tt L} &-n-\alpha-1&2n+\alpha+1&-n& \dfrac{n!\Gamma^2(\alpha+1)}
{\Gamma(n+\alpha+1)}\\
{\tt C} &-a&n+a&-n& \dfrac{n!}{a^n}\\[3pt]
{\tt M} &\dfrac{c(n+\beta)}{c-1}&-\dfrac{n+c(n+\beta)}{c-1}&\dfrac{n}{c-1}&
\dfrac{n!\Gamma(\beta)}{\Gamma(\beta+n)c^n}\\[3pt] 
{\tt M} &\dfrac{n+\beta}{c-1}&-\beta-\alpha_n-\gamma_n&\dfrac{n c}{c-1}&
\dfrac{n!\Gamma(\beta)c^n}{\Gamma(\beta+n)}\\[3pt] \midrule 
{\tt bqL}&(1\!-\!a q^{n+1})(1\!-\!b q^{n+1})&1-\alpha_n-\gamma_n&a b q^{n+1}(q^n\!-\!1)
& \dfrac{(q^{-1};q^{-1})_n\, q^{n}}{(a^{-1}q^{-1},b^{-1}q^{-1};q^{-1})_n}\\[2mm]
{\tt qM}&\dfrac{(c+q^{n+1})(b q^{n+1}-1)}{q^{2n+1}}&bq-\alpha_n-\gamma_n
&\dfrac{cq(q^n-1)}{q^{2n+1}}&\dfrac{(q;q)_n}{(bq,-c^{-1}q;q)_n\, q^n}\\[2mm]
{\tt lqL}&q^n(a q^{n+1}-1)&-\alpha_n-\gamma_n&a q^n(q^n-1) &\dfrac{(q;q)_n a^nq^n}{(aq;q)_n}\\[2mm]
{\tt qL}&\dfrac{q^{n+1+\alpha}-1}{q^{2n+1+\alpha}}&-\alpha_n-\gamma_n&\dfrac{q(q^n-1)}
{q^{2n+\alpha+1}}&\dfrac{(q;q)_n}{(q^{(\alpha+1)};q)q^n}\\[2mm]
{\tt qC}&-\dfrac{a+q^{n+1}}{q^{2n+1}}&-\alpha_n-\gamma_n&\dfrac{a q(q^n-1)}{q^{2n+1}}
&\dfrac{(q;q)_n}{(-a^{-1}q,q)_nq^n}\\[2mm]
{\tt 0LB}&-aq^{2n+1}&-\alpha_n-\gamma_n&-a q^n(q^n-1)&\dfrac{(q,q)_n a^n q^n}{(a q,q)_n}\\[2mm]
{\tt SW}&\dfrac{-1}{q^{2n+1}}&-\alpha_n-\gamma_n&\dfrac {q(q^n-1)}{q^{2n+1}}
&\dfrac{(q;q)_n}{q^n}\\ \bottomrule
\end{array}
\]
\caption{\label{table:1}Essential data of the families belonging to the Laguerre constellation.}
\end{table}

\appendix
\section{Code in wolfram Mathematica${}^\circledR$ to check the identities}
We are going to present different codes we can use to check all the 
identities presented in this work. Let us start with expression \eqref{lownL}:

\begin{lstlisting}[language=Mathematica]
Table[LaguerreL[n,a,x]-LaguerreL[n-1,a,x]-LaguerreL[n,a-1,x],{n,1,6}]//Factor
\end{lstlisting} 
One can check this for all the terms as one desires; using similar expressions
one can check 
\eqref{raidoaL}, \eqref{raiupaL} and the rest of the expressions where 
one uses the operators $\dln$ or $\Delta_n$.

In an analogous way one can check, for instance, \eqref{lowbM}:
\begin{lstlisting}[language=Mathematica]
mep[n_,b_,c_,x_]:=Hypergeometric2F1[-n,-x,b,1-1/c];Table[mep[n,b,c,x]-mep[n,b-1,c,x]-x n(c-1)/(b-b^2)/c mep[n-1,b+1,c,x-1],{n,1,6}]//Factor
\end{lstlisting} 

One can check identity used in the proof Theorem \ref{thm:3.5}:
\begin{lstlisting}[language=Mathematica]
chp[n_,a_,x_]:=HypergeometricPFQ[{-n,-x},{},-1/a];Table[a^{n+1}/(n+1)! chp[n+1,a,x]-a^n/n! chp[n,a,x]-a^{n+1}/(n+1)! chp[n+1,a,x+1],{n,1,6}]//Factor
\end{lstlisting} 

We also used different codes to know what kind of expression we expect 
\begin{lstlisting}[language=Mathematica]
m:=4; Solve[CoefficientList[mep[m+1,b,c,x]-mep[m, b, c, x]-(Ax+B)mep[m,bb,c,x-dd],x]==Table[0,{i,1,m+2}],{A,B,bb,dd}]
\end{lstlisting} 	
obtaining in such a case
\begin{lstlisting}[language=Mathematica]
{{A ->(-1+c)/(bc), B->0, bb->1+b, dd->1}}
\end{lstlisting}
Let me finish this sequence of expressions with the only one challenging
related with the big $q$-Laguerre polynomial expression, proposition 
\ref{pro:3}:
\begin{lstlisting}[language=Mathematica]
dfqu[k_,f_]:=If[k==0,f,If[k==1,Return[((f/.{x->q x})-f)/(qx-x)],Return[dfqu[k-1,dfqu[1,f]]]]];
bqLp[n_,a_,b_,x_]:=QHypergeometricPFQ[{q^{-n},0,x},{a q,b q},q,q];
AAbqL[k_,a_,b_,x_]:=A[k]bqLp[k+1,aq,bq,Bx]-Sum[AAbqL[l,a,b,x]dfqu[l,bqLp[k,a,b,x]],{l,0,k-1}])/dfqu[k,bqLp[k,a,b,x]]//FunctionExpand// Factor
\end{lstlisting}
By using the Solve \& coefficientList former algorithm modified 
to be used for the big $q$-Laguerre polynomials one used previously one 
can obtain the leading term $A[k]$, and assuming the expression should be 
finite is not hard to obtain $B\to q^4$.

\end{document}